\documentclass[a4paper,12pt]{amsart}
\usepackage{color}
\usepackage{amsmath}
\usepackage{amssymb}
\usepackage{amsthm}
\usepackage[mathscr]{eucal}

\numberwithin{equation}{section}

\newcommand{\ANR}{\operatorname{ANR}}
\newcommand{\CAT}{\operatorname{CAT}}

\newcommand{\R}{\operatorname{\mathbb{R}}}

\newcommand{\Sph}{\operatorname{\mathbb{S}}}

\theoremstyle{plain}
\newtheorem{thm}{Theorem}[section]
\newtheorem{lem}[thm]{Lemma}
\newtheorem{prop}[thm]{Proposition}
\newtheorem{cor}[thm]{Corollary}

\theoremstyle{definition}
\newtheorem{defn}[thm]{Definition}

\theoremstyle{remark}

\newtheorem{exmp}[thm]{Example}

\begin{document}

\title
[Regularity
of spaces with an upper curvature bound]
{Topological regularity of spaces \\
with an upper curvature bound}

\author
[A. Lytchak]{Alexander Lytchak}
\author
[K. Nagano]{Koichi Nagano}

\thanks{The first author was partially supported by the DFG grants   SFB TRR 191 and SPP 2026.
The second author was partially supported
by JSPS KAKENHI Grant Numbers 26610012, 21740036, 18740023,
and by the 2004 JSPS Postdoctoral Fellowships for Research Abroad.}

\email
[Alexander Lytchak]
{alytchak@math.uni-koeln.de}
\email
[Koichi Nagano]
{nagano@math.tsukuba.ac.jp}

\address
[Alexander Lytchak]
{\endgraf
Mathematisches Institut, Universit\"{a}t K\"oln
\endgraf
Weyertal 86-90, D-50937 K\"oln, Germany}
\address
[Koichi Nagano]
{\endgraf
Institute of Mathematics, University of Tsukuba
\endgraf
Tennodai 1-1-1, Tsukuba, Ibaraki, 305-8571, Japan}




\begin{abstract}
We   prove  that a locally compact space with an upper curvature bound is a topological manifold
if and only if all of its spaces of directions are homotopy equivalent and not contractible.
We discuss applications to homology manifolds,  limits of Riemannian manifolds and deduce a sphere theorem.
\end{abstract}

\maketitle
\section{Introduction}
\subsection{Main results}
 We prove the following
\begin{thm}\label{thm: topreg}
Let $X$ be a  connected, locally compact metric space  with an upper curvature bound.
Then
the following are equivalent:
\begin{enumerate}
\item
$X$ is a topological manifold.
\item
All tangent spaces $T_pX$ of $X$ are  homeomorphic  to the same space $T$, and $T$ is of finite topological dimension.
\item
All spaces of directions $\Sigma _pX$ are homotopy equivalent to the same space $\Sigma$, and $\Sigma $ is  non-contractible.
\end{enumerate}
\end{thm}

Theorem \ref{thm: topreg} answers the folklore question about the infinitesimal characterization of  topological manifolds among spaces with upper curvature bounds, compare \cite{Berest6}.
It implies an affirmative answer to a question of 
 F. Quinn,  \cite[Problem 7.2]{quinn-problem}:
\begin{thm} \label{thm: quinn}
	Let $X$ be a  metric space  with an upper curvature bound.  If $X$ is a homology manifold
	then  there exists a locally finite subset $E$ of $X$ such that $X\setminus E$ is a topological manifold.
\end{thm}


We refer the reader  to  \cite{Mio} and to Subsection \ref{subsec: homman} below for  basics on  homology manifolds
and to Section \ref{sec: topreg} for a  stronger  result.

 If $X$ in Theorem \ref{thm: topreg}
 is a topological  manifold of dimension $n$   then all tangent spaces $T_pX$ turn out to be  homeomorphic to $\R^n$ and all
spaces of directions turn out to be  homotopy equivalent to $\mathbb S^{n-1}$.  

  For $n\geq 5$,  the spaces of directions may not all be homeomorphic to $\mathbb S^{n-1}$,
 \cite{Berest3}, as a consequence of the double suspension theorem of R. Edward  \cite{E2}, \cite{Cannon}.  However,
 for $n\leq 4$, all spaces of directions  $\Sigma _p X$ are homeomorphic to $\Sph ^{n-1}$, see Theorem \ref{thm: dim3}. This 
answers  a question  of V.  Berestovskii \cite[Problem 1]{Berest5}.






We  deduce the following topological stability  theorem:

\begin{thm}\label{thm: riemlimreg}
For  $\kappa \in \R$ and $r > 0$,
let a sequence of
complete  $n$-dimensional Riemannian manifolds $M_i$
with sectional curvature $\le \kappa$
and injectivity radius $\ge r$ converge   in the pointed Gromov--Hausdorff topology to a locally compact space $X$.  Then
 $X$ is a topological manifold
  and any space of directions $\Sigma _x X$
of $X$ is homeomorphic to $\Sph ^{n-1}$.

 Moreover, if $X$ is compact then $M_i$ is homeomorphic to $X$, for all $i$ large enough.
\end{thm}

In particular, the  double suspension of a non-simply connected homology sphere, Example \ref{exmp: double},  is not a limit of $\CAT(\kappa)$ Riemannian manifolds, proving the conjecture formulated in  \cite{Berest3}.   

\subsection{Analogies and differences}
For spaces with lower curvature bounds the analogs of Theorem \ref{thm: topreg}  and the  stability part of Theorem \ref{thm: riemlimreg} are  special cases of the fundamental
topological stability theorem of G. Perelman, \cite{P2}, \cite{Kap}. Moreover, 
Theorem \ref{thm: quinn} for Alexandrov spaces is a direct consequence
of  Perelman's stability theorem, as observed in \cite{jywu}.
The analog of the additional statement in Theorem \ref{thm: riemlimreg}  about the homeomorphism type of the spaces of directions
(see also Theorem \ref{thm: part1} below, for a more general statement)  has been proved  for Alexandrov spaces by V. Kapovitch in \cite{Kap-reg}.

However, for spaces with an upper curvature bound  there is no analog
of the stability theorem,  even for  finite graphs. Moreover, already in dimension $2$, locally compact, geodesically complete spaces
with an upper curvature bound do not need to admit a topological triangulation, as has been observed by  B.  Kleiner, \cite{Kleiner}. Thus, unlike their analogs for Alexandrov spaces, our results are not special cases of much more general statements. 

 On the other hand, our approach requires less geometric control and should be applicable beyond our setting. For instance, it might simplify Perelman's stability theorem for Alexandrov spaces. 

As in  Perelman's topological theory  of Alexandrov spaces, a major role in our topological results play the so-called \emph{strainer maps}  investigated in \cite{LN}.
Perelman has proved in \cite{P2}, that in the realm of Alexandrov spaces strainer maps are local fiber bundles.
Similarly to the failure of topological stability, the  example  in \cite{Kleiner}  demonstrates  that in spaces with upper curvature bounds the local fiber bundle structure can not be expected. 
Nevertheless, from the homotopy point of view,
strainer  maps behave  well and turn out to be (local) \emph{Hurewicz fibrations}. This result,   Theorem \ref{thm: fibration}, is deduced from general topological statements
and the local contractibility of fibers of strainer maps   obtained in \cite{LN}.  Theorem \ref{thm: fibration}  might be useful in further investigations of spaces with upper curvature bounds and beyond.


We further mention, that the main theorems of  \cite{GPW}, (\cite{GPW-erratum}), \cite{Ferryfinite}  
imply  (in a more general situation) the finiteness of topological types of manifolds
in the sequence appearing in  the final statement of Theorem \ref{thm: riemlimreg}. However, no conclusion about the limit space itself can be 
deduced in the generality of \cite{GPW}, \cite{Ferryfinite} besides the fact that the limit space is a homology manifold.

Finally,  Theorems \ref{thm: topreg}, \ref{thm: quinn} in dimensions $\leq 3$ and some related insights in dimension $4$ are due to P. Thurston,
\cite{thurston}.

\subsection{Two applications}
In order to state  yet another manifold characterization 
we recall, \cite{LN},  that a space $X$ with an upper curvature bound is   \emph{locally geodesically complete} if any local geodesic $\gamma :[a,b]\to X$ can be extended as a local geodesic to some larger interval $[a-\epsilon, a+\epsilon]$.
 All homology manifolds, thus  all spaces appearing in the previous theorems, are always  locally geodesically complete, Lemma \ref{lem: busemann}.

\begin{thm} \label{thm: newchar}
Let $X$ be a  connected, locally compact space which has an upper curvature bound and is locally  geodesically complete.  If $X$ is not a topological manifold then it contains an isometrically embedded compact
metric tree different from an interval.
\end{thm}

 Theorem \ref{thm: newchar}  states that a non-manifold must have geodesics
which branch at an angle at least $\pi$. It can be seen as a soft version of the following much more special
and rigid result. If a connected, locally compact space $X$  with an upper curvature bound  is locally  geodesically complete  and has no branching geodesics then $X$ is a  smooth manifold whose distance
is defined by a continuous Riemannian metric   $g$  (with some additional properties), \cite{bermanifold}, \cite[Theorem 1.3]{LN}.

Theorem \ref{thm: newchar} is a  consequence of  Theorem  \ref{thm: topreg} and the following
sphere theorem.
\begin{thm} \label{thm: sphere}
Let $\Sigma $ be a compact,  locally geodesically complete  space with curvature bounded from above by $1$.
If the  injectivity radius of $\Sigma $ is at least $\pi$ and $\Sigma $ does not contain a triple of points with pairwise distances at least $\pi$ then $\Sigma $ is homeomorphic to a sphere.
\end{thm}



Our Theorem \ref{thm: sphere}  has a well-known analog for spaces with lower curvature bounds,  due to K. Grove  and P. Petersen, later reproved by A. Petrunin:
An Alexandrov space of curvature at least $1$  and \emph{radius} larger than $\frac \pi 2$  is homeomorphic to a sphere,
 \cite{grovpete}, \cite{petrunin}.   In terms of the \emph{packing radii} investigated in  \cite{GM}, \cite{GW2},
 the assumption about the triple of points in $X$ reads as  $\mathrm{pack} _3 (X) < \frac \pi 2$.
 
  From Theorem \ref{thm: sphere} it is easy to deduce a volume sphere theorem, see
 Theorem \ref{thm: volumesph} below,  generalizing \cite{coghitok}, \cite{knagano3}.

\subsection{Structure of the paper}
After Preliminaries in Section \ref{sec: prel}, we study in Section \ref{sec: homo}
homology manifolds with upper curvature bounds and prove those parts of our  main theorems,
which do not rely on properties of strainer maps.  In Section \ref{sec: prel2} we recall several topological results
  relating fibrations, fiber bundles and local uniform contractibility.  In Section \ref{sec: strainer} we recall from \cite{LN}
  basic properties of strainer maps and apply results from Section \ref{sec: prel2} to deduce Theorem \ref{thm: fibration} discussed above.
In Section \ref{sec: topreg}, we apply the general topological statements inductively to strainer maps and prove
Theorems \ref{thm: topreg}, \ref{thm: quinn}.  In Section \ref{sec: limit} we discuss iterated spaces of directions 
and prove Theorem \ref{thm: riemlimreg} and its generalization.  In the final Section \ref{sec: last} we discuss basic properties of \emph{pure-dimensional spaces} and 
derive the proofs of Theorems \ref{thm: newchar}, \ref{thm: sphere}, \ref{thm: volumesph}.

\subsection*{Acknowledgments} The authors would like to express their gratitude to  Steve Ferry for providing an elegant proof of  the $5$-dimensional case of  Theorem \ref{thm: addendum}. We  thank
Valeriy Berestovskii and Anton Petrunin for helpful comments.

\section{Preliminaries} \label{sec: prel}
\subsection{Notations}
 We refer to 
\cite{ABN}, \cite{BH}, 
 \cite{BuyaloSchr}, \cite{AKP}, \cite{LN}
for the basics
on  upper curvature bounds
in the sense of Alexandrov. 

We will stick to the  following notations.   By $d$ we denote the distance functions on metric spaces.
For a point $p$ in a metric space $X$, we denote by $d_x:X\to \R$ the distance function to the point $x$.
By  $B_r (p)$ (respectively, by $\bar B_r (p)$)
we denote the open (respectively, closed) metric ball of radius $r$ around the point  $p$. By $B_r ^* (p)$ we will denote the punctured ball $B_r (p) \setminus \{p \}$.
By $S_r (p)$ we will denote the metric sphere of radius $r$ around the point $p$.
Geodesics will always be globally minimizing and parametrized by arclength.

 For two  maps $\varphi, \psi$ from a set $G$ into a  metric space $Y$, we set
\[
d(\varphi,\psi) := \sup_{x \in G} d \bigl( \varphi(x),\psi(x) \bigr).
\] 
The dimension of a metric space $X$ will always be the covering dimension and will be denoted by $\dim (X)$.

A metric space  $X$ has curvature bounded from above by $\kappa$
if every point of $X$ has a $\CAT(\kappa )$ neighborhood.  If $X$ is a space with an upper curvature bound and $x\in X$ a point,
 we denote by $\Sigma _x$ or by $\Sigma_xX$ the \emph{space of directions} at $x$ and by $T_x$ or by $T_x X$ the \emph{tangent cone} at $x$ of $X$,  which is canonically identified with the
Euclidean cone over $\Sigma _x$.

We  denote by $H_k (X,Y)$ the  $k$-th singular homology  with integer coefficients
of the pair $Y\subset X$ of topological spaces.

\subsection{Basic topological properties of spaces with upper curvature bounds}
Any space $X$ with an upper curvature  bound is an \emph{absolute neighborhood retract}, abbreviated as 
$\ANR$, \cite{Ontaneda}, \cite{Linus}.
In particular, $X$ is  homotopy equivalent to a simplicial complex. We have,
 \cite{Linus}:
\begin{lem}\label{prop: whe}
For any point $x$ in a space $X$ with an upper curvature bound, there exists some $r>0$
such that
for each $0 < s \le r$, the ball $B_s (p)$ is contractible and the punctured ball $B_s ^* (p)$
is homotopy equivalent to
the space of directions  $\Sigma_p $.
\end{lem}

Due to \cite{Kleiner},
for any separable space $X$ with an upper curvature bound, we have
$$\dim (X) = 1+\sup _{x\in X} (\dim (\Sigma _x X))  =\sup _{x\in X} (\dim (T_x X))  \, .$$
Moreover, if  $\dim (X)$ is a finite number $n$, then there exists some $x\in X$
such that $H_{n-1} (\Sigma _x X)$ is not $0$.

\subsection{Tiny balls, GCBA spaces}
Let  $X$ be a locally compact space with an upper curvature bound $\kappa$. We will
say that a ball $B_r (x)$   is \emph{tiny} if  the closed ball with  radius $10\cdot r$ around $x$ is a compact $\CAT (\kappa )$ space and if
$100 \cdot r$ is smaller than the diameter of the
simply connected, complete  surface of constant curvature $\kappa$.

In  a tiny  ball all geodesics are 
determined by their endpoints. 

 A space $X$ with an upper curvature bound  is \emph{locally geodesically complete},
if every local geodesic defined on any compact interval can be extended as a local geodesic beyond its endpoints.
The following observation,  \cite[Theorem 1.5]{LS}, goes back to H. Busemann:
 \begin{lem} \label{lem: busemann}
 Let $X$ be a space with an upper curvature bound.
 If for any  $x\in X$ there exist arbitrary small $r$ such that the punctured ball $B^* _r (x)$ is non-contractible
 then $X$ is locally geodesically complete.
\end{lem}


Due to  the long exact sequence and the contractibility of small balls,
the \emph{local homology} $H_m(X,X \setminus \{x \} )$  at $x$ coincides with
$H_{m-1} (B_r ^{\ast} (x))$, for any $x$ in a space  $X$ with an upper curvature bound,
 any   small $r>0$ and any natural $m$.
 Thus, the non-vanishing of  $H_{\ast} (X,X\setminus \{x \})$, for all $x\in X$ implies, that $X$ is locally geodesically complete.

As a  \emph{GCBA space} we denote a  locally compact, separable metric space with an upper curvature bound, which is locally geodesically complete.
If $X$ is GCBA then so is any tangent space $T_x X$ and any space of directions $\Sigma _x X $. Moreover, any space of directions $\Sigma _x X$
is  compact and any tangent space $T_x X$ is the limit in the pointed Gromov--Hausdorff topology of  rescaled balls around $x$,
\cite{ABN}, \cite[Section 5]{LN}.


Every GCBA space $X$ has locally finite dimension, \cite[Theorem 1.1]{LN}, and
contains an open and dense  topological manifold (possibly of non-constant dimension),  \cite[Theorem 1.2]{LN}.
 Moreover, $X$ contains  a dense set of points with tangent spaces isometric to Euclidean spaces, possibly of different dimension,
 \cite[Theorem 1.3]{LN}.

The following result has been shown in \cite[Theorems 1.12, 13.1]{LN}:
\begin{prop} \label{prop: sph}
Let $x$ be a point in  a GCBA space $X$. Then there exists some $r_x>0$ such that for all $r<r_x$
the following hold true:
\begin{enumerate}
\item The metric sphere $S_r (x)$ is  homotopy equivalent to $\Sigma _x$.
\item  Let $B_{10 \cdot r} (x_i )$  be a  sequence of tiny  metric balls in GCBA spaces $X_i$ with the same upper curvature bound $\kappa$.
If $\bar B_{10 \cdot r} (x_i)$ converge to $\bar B_{10 \cdot r} (x)$ in the Gromov--Hausdorff topology then, for all $i$ large enough,
$S_r (x_i)$ is homotopy equivalent to $S_r (x)$.
\end{enumerate}
\end{prop}


\subsection{Homology manifolds} \label{subsec: homman}
 We denote by $D^n$ the closed unit ball in $\R^n$. We call $D^2$ the \emph{unit disk} and denote it by $D$.

Let $M$ be a locally compact, separable metric space
of finite topological dimension.
We say that
$M$ is a
{\em homology $n$-manifold with boundary}
if
for any $p \in M$ we have a point $x \in D^n$
such that
the local homology
$H_{\ast}(M,M \setminus \{ p \})$ at $p$
is isomorphic to $H_{\ast}(D^n,D^n \setminus \{ x \})$.
The {\em boundary} $\partial M$  of $M$
is defined as the set of all points
at which the $n$-th local homologies are trivial.
In the case where the boundary of $M$ is empty,
we simply say that
$M$ is a
{\em homology $n$-manifold}.

If $M$ is a homology $n$-manifold with boundary then $\partial M$
is a closed subset of $M$ and it is  a homology $(n-1)$-manifold
by \cite{Mitch}.

Any homology $n$-manifold (with boundary) has dimension $n$. For $n\leq 2$,  we have   the   theorem of R. Moore,
 see \cite[Chapter IX]{rwilder}.
\begin{thm} \label{thm: bing}
Any homology $n$-manifold with $n\leq 2$ is a topological manifold.
\end{thm}




\subsection{Examples}

The following example is well-known,  \cite{Berest2},\cite{grovpete}.

\begin{exmp}\label{exmp: double}
Consider a closed Riemannian $(n-2)$-manifold $Y$
which has the homology of $\Sph ^{n-2}$  but is not simply connected
(such manifolds exists in all $n\geq 5$).  Rescaling the metric
we may assume that $Y$  is $\CAT(1)$.
The spherical suspension $X_1=\Sph^0 \ast Y$ of $Y$ is a $\CAT(1)$ space which is a homology manifold and has
exactly two non-manifold points.  The double suspension $X=\Sph ^1 \ast Y$ of $Y$ is a $\CAT(1)$ space homeomorphic to $\mathbb S^n$
by the  double suspension theorem, \cite{Cannon}, \cite{E2}. But for any point $x$ on the $\Sph ^1$-factor, the space of directions $\Sigma _xX$
is isometric to $X_1$, hence not homeomorphic to $\Sph ^{n-1}$.
\end{exmp}

Some additional assumption on the tangent spaces and spaces of directions are needed in Theorem \ref{thm: topreg}:
\begin{exmp}
Let $X$ be the Hilbert cube, hence a compact  $\CAT (0)$ space. At any  $x\in X$ the space of direction $\Sigma _x$ is contractible. Moreover,
at any $x\in X$ the tangent space $T_x$  is homeomorphic to the Hilbert space, as can be deduced from \cite{Bestvina}, \cite{Tor}.
\end{exmp}


\section{Homology manifolds with upper curvature bounds} \label{sec: homo}
\subsection{General observations}  We start with the following:
\begin{lem} \label{lem: homchar}
Let $X$ be a locally compact space with an upper curvature bound.  $X$ is a homology $n$-manifold
if and only if $H_{\ast} (\Sigma _x) =H_{\ast} (\Sph ^{n-1})$, for all $x\in X$.   In this case, $X$ is locally geodesically complete.
\end{lem}

\begin{proof}
For any $x\in X$  and  $r>0$ as in Lemma \ref{prop: whe},
$\Sigma _x$ is homotopy equivalent to $B_r ^{\ast} (x)$ and $\tilde H_{k} (B_r^{\ast}  (x)) =H_{k+1}(X,X\setminus \{x\})$, for any $k$.

Thus,  if $X$ is a homology $n$-manifold, then  any space of directions $\Sigma _x X$ has the homology of $\Sph^{n-1}$.

 If all spaces of directions have the homology of $\Sph^{n-1}$ then they are non-contractible and, therefore,
 $X$ is locally geodesically complete.
 For any $x \in X$, we have   $H_{m}(X,X\setminus \{ x \}) =H_m (\Sph ^{n-1}) $.  Thus, $X$ has the same local homology
as $\R ^n$.   Since $X$, as any GCBA space, has locally finite dimension, $X$ is locally a homology $n$-manifold. Thus $X$ has topological dimension $n$ and it is a homology $n$-manifold.
\end{proof}

 Using the  contraction along geodesics to the center of a ball, \cite{Mitch}
and the Poincar\'e duality, \cite{Bredon},  
 we obtain
 \cite[Proposition 2.7]{thurston}:
 \begin{lem} \label{lem: spherehomology}
 Let $B_r (x)$ be a tiny ball in a homology $n$-manifold $X$ with an upper curvature bound.
 Then $\bar B_r (x)$   is a  compact, contractible homology manifold with boundary $S_r (x)$. In particular,
 $S_r (x)$ is a homology $(n-1)$-manifold with the same homology  as $\mathbb S^{n-1}$.
 \end{lem}

\subsection{Stability under convergence}
 The following observation is (essentially) a special case of result of E. Begle, \cite{Begle}, see also \cite[Theorem 2.1]{GPW}.  In our situation the proof can be  simplified using the
homotopy properties of distance spheres.

 \begin{lem} \label{lem: stable}
 Let $X_i$ be compact $\CAT(\kappa )$ spaces converging in the Gromov--Hausdorff topology to a compact space $X$.
Let $x_i \in X_i$ converge to $x\in X$ and let $r>0$ be such that for all $i$, the ball $B_r (x_i) \subset X_i$ is a homology $n$-manifold.
Then  $B_r(x)\subset X$  is a homology $n$-manifold.
 \end{lem}

\begin{proof}
Due to Lemma \ref{lem: homchar},  the open balls $B_r(x_i)$ are GCBA spaces, hence so is
$B_r (x)$, compare \cite[Example 4.3]{LN}. In particular, the  dimension of $B_r(x)$ is locally finite.

It remains to prove that $H_{\ast} (\Sigma _z X) =H_{\ast} (\mathbb S^{n-1})$ for all $z\in X$.
Write $z$ as a limit of points $z_i \in B_r (x_i)$. By Proposition \ref{prop:  sph} we find some $r>t>0$,
such that  $S_t (z) $  is homotopy equivalent to $\Sigma _z X$ and to $S_t (z_i)$, for all $i$ large enough.
Thus, the homology of $\Sigma _z X$ coincides with the homology of $\mathbb S^{n-1}$ by
 Lemma \ref{lem: spherehomology}.
\end{proof}

As a consequence we deduce:
\begin{cor} \label{cor: homlink}
Let $X$ be a homology $n$-manifold with an upper curvature bound.  Then, for any $x\in  X$, the tangent
space $T_xX$ is a homology $n$-manifold and $\Sigma _x$ is a homology $(n-1)$-manifold.
\end{cor}

\begin{proof}
The space $X$ is GCBA by Lemma \ref{lem: homchar}. Thus, the tangent space $(T_x X,0)$ is a limit of rescaled metric balls around $x$  in $X$. Due to Lemma \ref{lem: stable},
 $T_x$ is a homology $n$-manifold.   The homology $n$-manifold  $T_x \setminus \{0 \}$ is  homeomorphic
 to $\Sigma _x \times \mathbb R$.
Therefore, $\Sigma _x$ is a homology $(n-1)$-manifold.
\end{proof}

\subsection{Simple implications in the main theorem}
We can already discuss the simple implications in our main theorem.
We start with the following (folklore) result, compare \cite[Proposition 3.12]{BB}:

\begin{lem} \label{lem: 1to3}
Let $X$ be a topological $n$-manifold with an upper curvature bound.
Then any space of directions $\Sigma _x$ is homotopy equivalent to $\Sph ^{n-1}$.
\end{lem}

\begin{proof}
 Any space of directions $\Sigma _x$ is an ANR and  compact homology $(n-1)$-manifold with the homology of $\Sph ^{n-1}$, 
 Lemma \ref{lem: homchar}, Corollary \ref{cor: homlink}.
 
 If $n\leq 3$ then  $\Sigma _x$ is
homeomorphic to $\Sph^{n-1}$ by Theorem \ref{thm: bing}.

 If $n\geq 3$ then $\Sigma _x$
by  Whitehead's theorem  it suffices to prove that $\Sigma _x$ is simply connected.

Consider a small neighborhood $U$ of $x$ homeomorphic to a Euclidean ball  and  small numbers $r_1,r_2 >0$ such that
$B_{r_1} (x) \subset U\subset B_{r_2} (x)$.  The inclusion $B^{\ast} _{r_1} (x)  \to B^{\ast} _{r_2} (x)$ is a homotopy equivalence factoring
through the  simply connected space $U\setminus \{ x\}$.  This implies  that all small  punctured balls
$B^{\ast} _{r} (x)$ are simply connected. Due to Lemma \ref{prop: whe}, $\Sigma _x$ is simply connected,
finishing the proof.
\end{proof}




\begin{lem} \label{lem: same}
Let $X$ be a locally compact space  with an upper curvature bound. Assume that all spaces of directions $\Sigma_x X$ are homotopy equivalent to the same
non-contractible space $\Sigma$.  Then $\Sigma $ is homotopy equivalent to  $\Sph ^{n-1}$, for some $n$, and $X$ is a homology manifold.
\end{lem}

\begin{proof}
By assumption, all spaces of directions are non-contractible. By Lemma \ref{lem: busemann}, $X$ is a GCBA space.  Any GCBA space has a point with space of directions isometric  to a sphere $\Sph ^{n-1}$, \cite[Theorem 1.3]{LN}.
Then, by assumption, all spaces of directions are homotopy equivalent to $\Sph ^{n-1}$. By Lemma \ref{lem: homchar}, $X$ must be a homology $n$-manifold.
\end{proof}

\begin{lem} \label{lem: 2to3}
Let $X$ be a locally compact space  with an upper curvature bound. Assume that all tangent spaces  $T_x X$ are homeomorphic
to the same finite-dimensional space $T$. Then $T$ is homeomorphic to $\mathbb R^n$, for some $n$, and all spaces of directions
are    homotopy equivalent to $\mathbb S^{n-1}$.
\end{lem}

\begin{proof}
For points $x,y\in X$ there is, by assumption, a homeomorphism $I:T_x=C(\Sigma _x) \to T_y =C(\Sigma _y)$.
Restricting $I$ to a large distance sphere in $T_x$ around the origin, we obtain an embedding $I:\Sigma _x  \to C(\Sigma_y) \setminus \{y \} =(0,\infty )\times \Sigma _y$.
Composing with the projection to the second factor we obtain a map $\hat I:\Sigma _x\to \Sigma _y$, and it is easy to see (using the cone structures of $T_x$ and $T_y$) that $\hat I$ is a weak homotopy equivalence.
Since the spaces of directions are $\ANR$s, $\hat I$ is  a homotopy equivalence.  Thus, all spaces of directions are homotopy equivalent.

Due to \cite{Kleiner}, $X$ has finite dimension $n$, equal to the dimension of $T$.  Then,  by \cite{Kleiner} there exists some $x$ such that
$\Sigma _x$ is not contractible. By Lemma \ref{lem: same}, there exists some $n$ such that all spaces of directions are homotopy equivalent to $\Sph^{n-1}$.

Moreover, $X$ is a homology $n$-manifold and  a GCBA space by Lemma \ref{lem: homchar}.
By \cite{LN} there exists a point $x\in X$
with tangent space isometric to  $\R  ^n$. Therefore, $T$ is  homeomorphic to $\mathbb R^n$.
\end{proof}

\section{Homotopy stability  and Hurewicz fibrations} \label{sec: prel2}
\subsection{Uniform local contractibility}
Following \cite{Petersen1}  we say that  a functions $\rho:[0,r_0) \to [0,\infty )$ is a \emph{contractibility function}
if $\rho (0)=0$, $\rho (t) \geq t$, for all $t\in[0,r_0)$,  and $\rho$ is continuous at $0$.

 \begin{defn} \label{defn: glob}
We say that a family  $\mathcal F$ of metric spaces is \emph{locally uniformly contractible}
if  there exists a contractibility function $\rho :[0,r_0) \to [0,\infty )$ such that for any space $X$ in the family $ \mathcal F$,
any point $x\in X$ and any $0<r<r_0$, the ball $B_r(x)$ is contractible within the ball $B_{\rho (r)} (x)$.
\end{defn}

For example, the family of all  $\CAT(\kappa )$ spaces is locally uniformly contractible with $\rho :[0, \frac {\pi }{\sqrt  \kappa} ) \to [0,\infty)$ being the identity map.

 A  compact, finite-dimensinal space is locally uniformly contractible if and only if it is  an $\ANR$. 

 We will use the  notion of \emph{$\epsilon$-equivalence}, \cite{ChapmanFerry}, a controlled version of homotopy equivalence.  A continuous map $f:X\to Y$ between metric spaces is called an $\epsilon$-equivalence
if there exists a continuous map $g:Y\to X$  with the following property. There exist homotopies $F$ and $G$ 
 of $f\circ g$ and $g\circ f$ to the   respective identity map of $Y$ and $X$ such that the $F$-flow line of any point
 in $Y$  and the $f$-image of the $G$-flow line of any point in $X$  has diameter less than $\epsilon$ in $Y$.

The following result is due to P. Petersen,   \cite[Theorem A]{Petersen1}:
\begin{thm} \label{thm: petersen}
For any $n,\epsilon >0$ and any family $\mathcal F$ of locally  uniformly contractible metric spaces of dimension at
most $n$,
there exists some $\delta  >0$ such that the following holds true. Any  pair of spaces $X,Y \in \mathcal F$,
with  Gromov--Hausdorff distance  at most $\delta $ are $\epsilon$-equivalent.
\end{thm}


When dealing with the family of   fibers of  a map the following variant of Definition \ref{defn: glob} seems  more suitable,
compare \cite{Ungar}.

\begin{defn} \label{defn: loc}
Let $F:X\to Y$ be a  map between metric spaces.
We say that  \emph{$F$ has locally uniformly contractible fibers}
if the following condition holds true for any point $x\in X$ and every neighborhood $U$ of $x$ in $X$.
There exists a neighborhood $V\subset U$ of $x$ in $X$ such that for any fiber
$F^{-1}(y)$ with non-empty intersection $F^{-1} (y)\cap V$, this intersection is contractible in $F^{-1} (y)\cap U$.
\end{defn}

 For    $X$ compact, a map $F:X\to Y$ has locally uniformly contractible fibers in the sense of Definition \ref{defn: loc}
if and only if the family of the fibers is locally uniformly contractible in the sense of Definition \ref{defn: glob}.

\subsection{Relation to Hurewicz fibrations}
A   map  $F:X\to Y$ between metric spaces is called   a \emph{Hurewicz fibration}
if it  satisfies the homotopy lifting property with respect to all spaces, \cite[Section 4.2]{Hatcher}, \cite{Ungar}.

The map $F$ is called open if the images of open sets are open. It is called proper if the preimage of any compact set is compact.

Any locally compact metric space carries a complete metric. This allows us to formulate  Theorems \ref{thm: selection},
\ref{thm: hurewicz}, \ref{thm: localstatement}  below for locally compact metric spaces, while the original formulations
in \cite{Michael2}, \cite{Ungar} are done for complete metric spaces.

We formulate a special case of the continuous selection theorems of E. Michael, \cite[Theorem 1.2]{Michael2}, as   in
in \cite[Theorem M]{Dyer}:
\begin{thm} \label{thm: selection}
Let $F:X\to Y$ be an open  map with locally uniformly contractible fibers between finite dimensional, locally compact metric spaces.
Then, for any  $x\in X$, there exist a neighborhood $U$ of $F(x)$ in $Y$ and a continuous map $s:U\to X$ such that $F\circ s$ is the identity.
\end{thm}



 The following result is proved in \cite[Theorem 1]{Ungar}, see also  \cite{addis}  and \cite{ferry-anr}   for related statements.
\begin{thm} \label{thm: hurewicz}
Let $X,Y$ be finite-dimensional,   compact metric spaces and let $Y$ be an  $\ANR$.
Let $F:X\to Y$ be an open, surjective   map   with locally uniformly contractible fibers.
Then $F$ is a Hurewicz fibration.
\end{thm}

In the locally compact case one can not expect that an open, surjective  map with locally uniformly contractible fibers
is  a Hurewicz fibration, as we see  by restricting a Hurewicz fibration to a complicated open subset.
However,
 the following result is
deduced in \cite[Theorem 2]{Ungar} from the selection theorem of Michael mentioned above.

\begin{thm} \label{thm: localstatement}
Let $X,Y$ be finite-dimensional locally compact metric spaces.
Assume that an open, surjective map $F:X\to Y$ has   locally uniformly continuous fibers.
If, in addition, all fibers $F^{-1} (y)$ of the map $F$ are contractible then $F$ is
  a Hurewicz fibration.
\end{thm}

\subsection{Fibrations and fiber bundles}
In some situations, Hurewicz fibrations turn out to be fiber bundles.  Most results  in this direction
are based  on the famous  $\alpha$-approximation theorem,  proved  by T. Chapman  and S. Ferry in dimensions $n\geq 5$, \cite{ChapmanFerry},
and  extended by S. Ferry and S. Weinberger to dimension $n=4$,  \cite[Theorem 4]{Ferry}, and by W. Jakobsche to dimensions $n=2,3$,   \cite{Jakobsche2}, \cite{Jakobsche}.
Note that the $3$-dimensional statement in \cite{Jakobsche} relies on the resolution of the  Poincar\'e conjecture. For $n=1$, the $\alpha$-approximation theorem  is rather clear.  

\begin{thm}  \label{thm: alpha}
Let the metric space $M$ be  a closed topological $n$-manifold.  For any  $\alpha >0$ there is some $\epsilon =\epsilon(M,\alpha) >0$ such that   for any closed topological  $n$-manifold $M'$  and any $\epsilon$-equivalence  $f:M'\to M$   there exists  a homeomorphism $f':M'\to M$ with $d(f,f')<\alpha$.
\end{thm}

This theorem combined with the fiber-bundle recognition developed in \cite{Dyer} implies 
 \cite[Theorems 1.1-1.4]{ferry-bundle}, \cite[Theorem 2]{Raymond-fibering}:

\begin{thm}  \label{thm: trivialbundle}
Let $X,Y$ be finite-dimensional locally compact $\ANR$s.
Let $F:X\to Y$ be a Hurewicz fibration. If all fibers of $F$ are closed $n$-manifolds
then $F$ is a locally trivial fiber bundle.
\end{thm}

We will also apply the following local  variant of this global result proved in \cite[Proposition 4.2]{ferry-bundle}.
The case $n=3$, excluded in  \cite[Proposition 4.2]{ferry-bundle}, need not be excluded due to the solution of the Poincar\'e conjecture (and \cite{Jakobsche}):
\begin{thm} \label{thm: localproduct}
Let $F:X\to I$ be a Hurewicz fibration from a metric space $X$ to an open interval $I$.
 Assume that all fibers  are topological $n$-manifolds.
Then
$X$ is a topological  $(n+1)$-manifold.
\end{thm}

\subsection{Fibrations and homology manifolds} Finally, we will use
the following result  proved by F. Raymond   in \cite[Theorem 1]{Raymond-fibering}.
Relying on the  the local orientability of homology $n$-manifolds \cite{Bredon},  Raymond's Theorem 1 can be slightly strengthened as explained in  \cite[p.52-53]{Raymond-fibering}.
(The result will be used   only for  Euclidean balls $Y$).
\begin{thm} \label{thm: raymond}
Let $X$ be a homology $n$-manifold and let  $F:X\to Y$ be a Hurewicz fibration.
If $Y$ is connected and  locally contractible then there exists some $k\leq n$ such that  any fiber of $F$ is a homology $(n-k)$-manifold and $Y$ is a homology $k$-manifold.
\end{thm}

\section{Strainer maps} \label{sec: strainer}
\subsection{Basic properties}
We  recall the basic properties of strainer maps in GCBA spaces, a tool invented in \cite{BGP} for Alexandrov space, and  applied to GCBA and investigated in this context in \cite{LN}.  We are not going  to recall the exact definition but state instead the
properties of strainer maps which will be used below.


Let $O$ be a tiny  ball   of a GCBA space $X$.
For  any natural $k\geq 0$, and any $\delta >0$ there is the family $\mathcal F_{k,\delta} =\mathcal F_{k,\delta} (O)$ of  the so-called \emph{$(k,\delta)$-strainer maps}
$F:U\to \R^k$ defined on open subsets $U$ of $O$ with the following properties, \cite[Sections 7, 8]{LN}.

(0) By convention, for $k=0$, any $\delta >0$ and any open $U\subset O$,
 we let the constant map $F:U\to  \{0 \} =\R^0$ be a $(0,\delta)$-strainer map.


(1)  For any $F\in \mathcal F_{k,\delta} (O)$, the coordinates $f_i$ of $F$ are  distance functions to some points $p_1,...,p_k \in O$.

(2)  For $\delta _1 >\delta  _2$,  we have the inclusion $\mathcal F_{k,\delta _2} (O) \subset\mathcal F_{k,\delta _1} (O)$.

(3)  For any  $F\in \mathcal F_{k,\delta} (O)$ and $l<k$,  the first $l$ coordinate functions of  $F:U\to \R^k$
define a map $\tilde F:U \to \R^l$ contained in  $\mathcal F_{l,\delta} (O)$.

(4) The restriction of any $(k,\delta)$-strainer map to any open subset is a $(k,\delta)$-strainer map.

\subsection{Extension properties} \label{subsec: extension}
All  extendability properties of strainer maps and the "largeness" of the sets $\mathcal F_{k,\delta }(O)$
depend on the following:

(5)  For any $x\in X$ and any $\delta >0$, there exists some $r>0$ such that $d_x:B_r ^{\ast} (x)\to \R$ is contained in $\mathcal F_{1,\delta}$; \cite[Proposition 7.3]{LN}.

This result  has  the following generalization, \cite[Proposition 9.4]{LN}.


(6)  Let $F:U\to \R^k$ be a map in $\mathcal F_{k,\delta}$.  Let $x\in U$ be a point and let $\Pi$ be the fiber $F^{-1} (F(x))$.      Then, there is $r>0$ and an  open set $V\subset U$ containing $B_r^{\ast} (x) \cap \Pi$,  such that
the map $\hat F=(F, f) :V\to \R^{k+1}$ with last coordinate $f=d_x$ is  contained in  $\mathcal F_{k+1,12\cdot \delta}$.

 This property (6) is the "fiber-wise" statement of the following closely related result, contained  in \cite[Theorem 10.5]{LN} in a stronger form:

(7)    Let   $F:U\to \R^k$  be in $\mathcal F_{k,\delta}$.  Consider the set $K$ of points $x\in U$ at which
$F$ can not be locally extended to a $(k+1,12\cdot \delta)$-map $\hat F=(F,f) :U_x \to \R^{k+1}$. Then the closed set  $K$
 intersects  any fiber of $F$ in $U$ in a finite set of points.  

\subsection{Topological properties}
The following 
property
 is contained in \cite[Theorem 1.10]{LN}:

(8) Let $F:U\to \R^k$ be a map in $\mathcal F_{k,\delta}$ with  $\delta < \frac 1 {20 \cdot k}$. Then the map $F$
is open. Moreover, for  any compact subset $K$ of $U$, there exists some $\epsilon >0$ such that for all $r<\epsilon$
and all $x\in K$ the intersection $B_r(x) \cap F^{-1} (F(x))$ is contractible.

Now we easily derive:

\begin{thm} \label{thm: fibration}
	Let $U$ be an open subset of a GCBA space  $X$. Let $F:U\to \R^k$ be a $(k,\delta)$-strainer map, for some  $k$ and
	any $\delta <\frac {1} {20 \cdot k}$.
	
	Then any $x\in U$ has arbitrary small  open contractible neighborhoods $V$, such that the restriction $F:V\to F(V)$ is a Hurewicz fibration with contractible fibers.
	
	If a fiber $F^{-1}  (b)$ is compact then there exists an open neighborhood $V$ of $F^{-1} (b)$ in $U$ such  that
	$F:V\to F(V)$ is a Hurewicz fibration.
\end{thm}

\begin{proof}
By the property (8) of strainer maps, the map $F$ is open and it has locally uniformly contractible fibers.

Let $x\in U$ be arbitrary. Using Theorem \ref{thm: selection} we find a neighborhood  $W$ of $F(x)$ in $\R^k$
and  a continuous section $s:W\to U$ such that  $F\circ s$ is the identity. Making $W$ smaller, if needed, we may assume
that $W$ is an open ball and $s(W)$ is contained in a compact subset $K\subset U$.

Take a positive number $\epsilon $ provided by the property (8). Making $\epsilon $ smaller, we may assume that the
distance from $K$ to the boundary of $U$ in $X$ is larger than  $\epsilon$.   Consider the set $V\subset U$ (the union of balls-in-the-fiber of radius $\epsilon$).
 of all
$z\in F^{-1}(W)$ such that $d(z, s(F(z)))<\epsilon$.

Then $V$ is open in $U$ and  contains $x$.  We have $F(V)=W$ and every fiber $F^{-1} (\mathfrak t) \cap V$ of $F$  in  $V$ is a contractible.
Applying Theorem \ref{thm: localstatement} we see that $F:V\to F(V)$ is a Hurewicz fibration. Since $W=F(V)$ and the fibers of the Hurewicz fibration $F:V\to W$ are contractible, $V$ is contractible as well.

Let now $F^{-1} (b)$ be a compact fiber of $F$ in $U$. Take a compact neighborhood $V_0$ of $F^{-1} (b)$ in $U$.
Let $C$ be its boundary $\partial V_0$. Consider a closed ball $B$  around $b$ which is contained in the neighborhood
$F(V_0)$ of $b$ but  does not intersect the compact image
$F(C)$.     Let $V _1$ be the intersection $V_0\cap F^{-1} (B)$.

Since $F^{-1} (B)$ does not intersect $C$, the set $V_1$ is compact.   The restriction $F:V_1\to B$ has locally uniformly contractible fibers.  Applying Theorem \ref{thm: hurewicz} we deduce that $F:V_1\to B$ is a Hurewicz fibration.    If we take $W$ to be any open ball around $b$ contained in $B$ and let $V$ be the preimage
$F^{-1} (B)  \cap V_1$ then $F:V\to B$ is a Hurewicz fibration as well.  This finishes the proof.
\end{proof}

Since being  a homology $k$-manifold is a local property,  we directly deduce from Theorem  \ref{thm: fibration}  and
 Theorem \ref{thm: raymond}:

\begin{cor}  \label{lem: admiss}
Let  $F:U\to \R^k$ be a $(k,\delta)$-strainer  map
  with $\delta < \frac 1 {20 \cdot k}$  defined on an open subset
  of a GCBA space $X$.
   If  $U$ is a homology $n$-manifold  then any  non-empty fiber    $\Pi$ of $F$ is a homology $(n-k)$-manifold.
\end{cor}

\section{Topological regularity} \label{sec: topreg}
\subsection{Disjoint disk property}

A metric space $M$
has the {\em disjoint disk property}
if
for any two continuous maps $\varphi_i \colon D \to M$, $i = 1, 2$, on the unit disk $D$ 
and for any $\epsilon > 0$,
there are two continuous maps  $\widetilde{\varphi}_i \colon D \to M$
such that
$d(\varphi_i,\widetilde{\varphi}_i) \le \epsilon$
and 
$\widetilde{\varphi}_1(D) \cap \widetilde{\varphi}_2(D) =\emptyset$.

For a homology $n$-manifold $Y$ we  denote by
 $\mathcal M(Y)$ the set of manifold points in $Y$, thus of all points  in $Y$ with a neighborhood homeomorphic to $\R^n$.
We recall the following special case of the celebrated manifold recognition theorem of Edward--Quinn,
 \cite[Theorem 2.7]{Mio}:

\begin{thm} \label{thm:  ddp}
Let the connected metric space $Y$ be an $\ANR$ and   a  homology $n$-manifold with $n\geq 5$. Then $Y$ is
a topological manifold if and only if the set of manifold points $\mathcal M (Y)$ is not empty and $Y$ has the disjoint disk property.
\end{thm}

For $n\geq 5$ the next result  easily follows from Theorem \ref{thm:  ddp} and is  a very special case of the main theorem of \cite{CBL}.
For $n=4$, the next result is a very special case of the main theorem of \cite{BestvinaDav}.
\begin{thm} \label{thm: bcl}
	Let $Y$ be an $\ANR$ and a homology $n$-manifold with $n\geq 4$. Let 
	$K\subset Y$ be a discrete set of points such that $Y\setminus K$ is a topological $n$-manifold. If every point $x\in K$ has  arbitrary small neighborhoods $U$ in $Y$ such that $U\setminus \{x\}$ is simply connected then $Y$ is a topological $n$-manifold.
\end{thm}

\subsection{Structure of GCBA homology manifolds}
We are going to formulate and  prove the main technical result.

\begin{thm} \label{thm: maintech}
For natural numbers  $0\leq k \leq n$,
let $U\subset X$ be an open subset of a GCBA space $X$. Assume that $U$ is a homology $n$-manifold.
Let $F:U\to  \R^{n-k}$  be an  $(n-k,\delta )$-strainer map and let  $\Pi$ be a fiber of the map $F$. Let $E\subset \Pi$ be
the  set of points at which $F$ does not have a local extension to
an $(n-k+1, 12 \cdot \delta )$-strainer map $\hat F =(F,f)$.

Assume finally that $\delta <20 ^{-n+k-1}$.    Then the set $E$ is finite and the complement
$\Pi \setminus E$ is a topological  $k$-manifold. Moreover, if $k\leq 3$ then $\Pi$ is a topological $k$-manifold.
\end{thm}

\begin{proof}
 By our assumption, $\delta < \frac 1 {20 \cdot  (n-k)}$ and $12\cdot \delta < \frac 1 {20 \cdot (n-k+1)}$. Thus,  Corollary \ref{lem: admiss} and  Theorem \ref{thm: fibration}  apply to $F$ and the extensions of $F$ provided by Subsection \ref{subsec: extension}.
Hence,  $\Pi$ is a homology $k$-manifold by Corollary \ref{lem: admiss}. Due to Subsection
\ref{subsec: extension}, the set $E\subset \Pi$  is finite.



We fix $n$ and proceed by induction on $k$.
 For $k\leq 2$, we deduce from   Theorem \ref{thm: bing} that $\Pi$ is  a topological $k$-manifold.

Assume  $k=3$. Let $x\in \Pi$ be arbitrary. By Subsection \ref{subsec: extension}, we find some $r>0$ such that the ball  $\bar B_r(x) \subset U$   is compact and has the following property.
There exists an open set $V\subset X$ containing $B_r^{\ast} (x) \cap \Pi$
such that  the map $\hat F=(F,f): V\to \R^{n-k+1} $ is an $(n-k+1,12\cdot \delta )$-strainer map,
where $f$ is the distance function $f=d_x$.

The fibers of $\hat F$ through points $z\in B_r^{\ast} (x) \cap \Pi$ are compact distance spheres $\Pi_t: =S_t (x) \cap \Pi$ around $x$ in $\Pi$.
By Theorem  \ref{thm: fibration} the restriction of $\hat F$ to a neighborhood of any such fiber $\Pi_t$ is a Hurewicz fibration.
Hence, the restriction of $f$   to a neighborhood of $\Pi_t$ in $\Pi$ is a Hurewicz fibration, for any $0<t<r$.
Therefore, $f:B_r^{\ast} (x) \cap \Pi \to (0,r)$  is a Hurewicz fibration.

By the already verified  case $k=2$, the fibers of $\hat F$ (hence of $f$) are topological $2$-manifolds. Due to Theorem \ref{thm: trivialbundle},
the Hurewicz fibration  $f$ must be a fiber bundle. Since the base of the bundle is a contractible interval, the bundle must be trivial. Thus,
$B_r^{\ast} (x) \cap \Pi$ is homeomorphic to $(0,r)\times M$ for  a  topological $2$-manifold $M$.

From the uniqueness of one-point compactifications, we see that  $B_r (x) \cap \Pi$ is homeomorphic to
the cone $CM$ over $M$. Since $\Pi$ is a homology $3$-manifold, $M$ must have the homology of $\Sph ^2$.
Therefore, $M$ is homeomorphic to $\Sph ^2$ and
 $B_r (x) \cap \Pi$ is homeomorphic to $\R^3$. Since  the point $x$ was arbitrary, $\Pi$ is a topological $3$-manifold.

Assume now $k=4$.  For any point $x\in \Pi \setminus E$, there exists a neighborhood $U_x$ of $x$ in $X$ and an extension
of $F$ to an $(n-k+1,12\cdot \delta )$-strainer map $\hat F=(F,f) :U_x \to \R^{n-k+1}$.  Applying the case $k=3$, we know that the
fibers of $\hat F$ are topological $3$-manifolds.     Due to Theorem \ref{thm: fibration}, we may restrict to a smaller neighborhood of $x$  and
assume that $\hat  F:U_x\to \hat F(U_x)$ is a  Hurewicz fibration. Then so is the restriction $\hat F : \Pi \cap U_x  \to  \hat F(\Pi \cap U_x)$, which is nothing else but the last coordinate $f$.   Applying Theorem \ref{thm: localproduct}, we  see that $\Pi \cap U_x$ is a $4$-manifold.   Since $x\in \Pi \setminus E$  was arbitrary, this finishes the proof for  $k=4$.

Assume now $k\geq 5$ and  that the claim is true for $k-1$.
 We consider an arbitrary point $x\in \Pi\setminus E$, a neighborhood $U_x$
of $x$ and an extension of $F$ to an $(n-k+1,12\cdot  \delta)$-strainer map $\hat F =(F,f)$  as before. Making $U_x$ smaller,
we may assume by Theorem \ref{thm: fibration}, that  the
  restriction  $\hat F: U_x\to \hat F (U_x)$ is a Hurewicz fibration with contractible fibers.

Consider the intersection $W := \Pi \cap U_x$ and the restriction $f: W\to f(W) \subset \R$ which is a Hurewicz fibration with contractible fibers.
 By making $U_x$ smaller (if needed), we may and will assume that $f(W)$ is an open interval $J\subset \R$.
 In this setting we will prove that $W$ is a topological $k$-manifold.

   For $t\in J$ we let  $W_t$ the preimage $f^{-1} (t) \subset W$, which is a  contractible fiber  of the strainer map $\hat F: U_x \to \R^{n-k+1}$.

Let $K_1$ be the closed subset of points of $U_x$ at which $\hat F$ does not (locally) extend to an $(n-k+2, (12)^2 \cdot \delta)$-strainer map.
By the inductive assumption, the  intersection of any fiber of $\hat F$ with $U_x\setminus K_1$ is a topological $(k-1)$-manifold.
Applying Theorem \ref{thm: localproduct} to the Hurewicz fibration $f:W\to J$ we deduce that $W\setminus K_1$ is a topological $k$-manifold.

Consider the set of manifold points $\mathcal M(W)$ and its complement $K_0:=W\setminus \mathcal M(W)$, the set of non-manifold points in $W$.
We have just shown that $K_0$ is contained in $K_1$.   We assume that $K_0$ is not empty, and we are going to derive a contradiction.

By the inductive assumption, the set $K_1$  intersects every fiber  of $\hat F$ only in finitely many points. Hence, for any $t\in J$ the intersection  $W_t \cap K_0$ is finite.

The Hurewicz fibration $f:W\to J$ has contractible base and fibers, hence $W$ is contractible, in particular, it is connected.
The set $W\setminus K_0$ is not empty, as we have seen.
 Due to Theorem \ref{thm:  ddp}, it suffices to prove  that $W$ satisfies the disjoint disk property, in order to conclude that
$W$ is a topological $k$-manifold and to achieve a contradiction.

 The verification of the disjoint disk property occupies the rest of the proof and  happens in several steps.

\emph{Step 1}:  For any map $\gamma :\Sph ^1 \to W$  and $\epsilon >0$ there exists a map  $\hat \gamma :\Sph^1 \to W$
with $d(\gamma, \hat \gamma ) <\epsilon$ such that
 $f\circ \hat \gamma $  is piecewise monotone.

Indeed, we easily find a homotopy of the map $\eta _0:=f\circ \gamma :\Sph^1 \to J$ through maps $\eta _t$ such that each  $\eta _{t}$ for $t>0$
is piecewise linear. Using that $f$ is a Hurewicz fibration we can lift $\eta _t$ to a homotopy of $\gamma =\gamma _0$. Then we find the required
map $\hat \gamma$ as $\gamma _t$ for a small $t$.

\emph{Step 2}: The set $\mathcal M(W)=W\setminus K_0$ is connected.

Indeed,   for any $t\in J$, the fiber $W_t$ is a connected homology $(k-1)$-manifold.
Since $W_t\cap K_0$ is discrete, the complement  $W_t\setminus K_0$ is not empty and connected, see
\cite[Lemma 2.1]{daverma1}.  For  any connected component  $W'$  of $\mathcal M(W)$   we deduce
$W' =f^{-1} (f(W'))  \cap \mathcal M(W)$. Since $J$ is connected, this implies $f(W')=J$ and $W'=\mathcal M(W)$.

  \emph{Step 3}:  For any $y\in W$, the complement $W \setminus \{y  \}$ is simply connected.

  Indeed, consider an arbitrary curve  $\gamma : \Sph ^1 \to W \setminus \{ y \}$.   In order  to fill $\gamma$  by a disk,
  we use the local contractibility of $W$ and Step 1, and may assume that $\eta = f\circ \gamma$ is piecewise  monotone.  If the image of $\eta $ does not contain $t_0 :=f(y)$ then $\gamma$ lies  in the contractible set $f^{-1} (\eta  (\Sph ^1))$  (which does not contain the point  $y$)  and the statement is clear.

  If the image of $\eta$ contains $t_0$, we can write $\eta$ as a concatenation of finitely many curves $\eta _i$ based in $t_0$, each of them completely contained either in  $[t_0,\infty )$ or in $(-\infty ,t_0]$.     The corresponding decomposition of $\Sph ^1$ decomposes $\gamma$ in a finite concatenation of possibly non-closed curves $\gamma _i$ each of them ending and starting on $W_{t_0}$.

  The homology $(k-1)$-manifold  $W_{t_0} $ is connected, hence so is $ W_{t_0} \setminus \{ y\}$. Therefore,
  we can connect  the endpoints of each $\gamma _i$ in $W_{t_0}$.

 Concatenating these "connection curves " with $\gamma$ we obtain a closed curve $\hat \gamma$, homotopy equivalent to
 $\gamma$ in $W\setminus \{ y\}$. Moreover, $\hat  \gamma$
  is  a concatenation of finitely many closed curves $\tilde \gamma$, such that $f\circ \tilde \gamma$ is contained  either in
   $(-\infty ,t_0]$ or $[t_0,\infty )$.

   For any such curve $\tilde \gamma$
   we can  now fill  $f\circ \tilde \gamma$ in $J$ by a disk none of whose interior point is sent to $t_0$. Using the homotopy lifting property, we can lift this disk to a filling of $\tilde \gamma$ in $W\setminus \{ y\}$. Thus, any of the curves $\tilde \gamma$ and hence $\gamma$ are contractible in $W\setminus \{ y\}$.

\emph{Step 4}: For  any curve $\gamma :\Sph ^1\to W \setminus K_0$, there exists an extension of $\gamma $ to a disk
$\phi : D\to W$ intersecting $K_0$ only in  finitely many points.

Indeed, arguing as in Step 3, we can assume that $f\circ \gamma$ is piecewise monotone.  Subdividing
$f\circ \gamma$  and using connection curves in single fibers of $f$, as in the previous step, we reduce the question to the case that $f\circ \gamma $ is the concatenation of two monotone curves.
Reparametrizing $\gamma$ we can assume that
$\gamma $ is parametrized on an interval $[-a,a]$ such that  $f\circ \gamma (q)=f\circ \gamma (-q)$ for all
$q\in [0,a]$.

For any $q\in [0,a]$ we  choose any  curve $\gamma _q$ in $W_{f(\gamma (q))} \setminus K_0$
connecting $\gamma (-q)$ and $\gamma (q)$.  Let $Q$ denote the set of numbers $q\in [0,a]$ such that the concatenation of $\gamma _q$ and  $\gamma |_{[-q,q]}$  can be filled by a disk in $W$ intersecting only finitely many points in  $K_0$.

Clearly $Q$ contains $0$. We are done if $Q$ contains $a$.
 Using a connectedness argument
it suffices to prove that for any $q_0$ there exists some $\epsilon >0$  such that  for any
$q$ with $|q-q_0|<\epsilon$ the concatenation  $\gamma _{q,q_0}$
of $\gamma _q, \gamma _{q_0}$ and the parts of $\gamma$ between $\pm q$ and $\pm q_0$ can be filled in $W$ by a disk intersecting $K_0$ only in a finite number of points.

We fix $q_0 \in J$.

Since  the Hurewicz fibration $f:W\to J$ has   contractible fibers,  we can find a continuous family $P_s, s\in J$
of homotopy retractions $P_s:W\times [0,1]$ from $W$ to $W_s$.  Indeed, the map $f$
satisfies the homotopy extension property for
every pair of finite-dimensional spaces, see  \cite[Theorem 1.2]{Michael2}.   Thus, we can
extend  a continuous  map $P:W\times J\times [0,1] \to W$ such that $P(w,f(w),t) =P(w,s,0)=w$ for all $w\in W, t\in [0,1]$
 and $s\in J$  and such that $f\circ P  (w,s,t) =(1-t)\cdot f(w) + t \cdot s$ for all $(w,s,t)\in W\times J\times [0,1]$.

By continuity we find some $\epsilon >0$ such that for all $q\in [0,a]$ with $|q-q_0|<\epsilon$ the
homotopy retraction $P_{f(\gamma (q))} $ from $W$ onto the fiber $W_{f(\gamma (q))}$ has the following property:
The trace under this homotopy retraction of $\gamma _{q_0}$ and both parts of $\gamma$ between
$\pm q_0$ and $\pm q$ do not intersect $K_0$.

Therefore, the homotopy  retraction $P_{f(\gamma (q))} $  defines a homotopy (not intersecting $K_0$) of the curve
$\gamma _{q,q_0}$ to a closed curve $c$ completely contained in the fiber $W_{f(\gamma (q))}$.  Filling the curve $c$
inside the contractible fiber $W_{f(\gamma (q))}$ by any disk, we obtain  the required filling of the curve $\gamma_{q,q_0}$.
This finishes the proof of Step 4.

Step 5:  For all $z\in W$  and all $\epsilon >0$ there exists an open contractible neighborhood $V^z$ of $z$ in $W$ with diameter smaller than $\epsilon$
such that  the restriction $f: V^z \to f(V^z)$ is a Hurewicz fibration with contractible fibers.

Indeed, this follows  from Theorem
\ref{thm: fibration} in the same way as in the construction of $W$.

\emph{Step 6}:  The conclusions of Step 3 and Step 4 are valid for all neighborhoods $V^z$ constructed in Step 5.

Indeed, the proofs of the respective steps apply literally.

\emph{Step 7}:   For every   disk $\phi : D\to W$ and every $\epsilon >0$,  there exists a disk $ \phi _{\epsilon}:  D\to W$ with pointwise distance to $\phi$ at most $\epsilon$ and such that the image of $ \phi _{\epsilon}$ meets $K_0$ only in a finite set.

 Indeed, we consider a covering of the $\phi ( D)$ by the sets $V^z$ described above each of them of diameter at most $\frac \epsilon 3$.
 Using the Lemma of Lebesgue  we find a triangulation of the disk $D$ by a finite graph $\Gamma$,
 such that for any  $2$-simplex  $\Delta$ of the triangulation, the image $\phi (\Delta)$ is contained in one of the sets $V^z$.

 We slightly move the images of the vertices of $\Gamma$ and use Step 2 and Step 6
 in order to  find  a map $ \phi_{\epsilon}  :\Gamma \to U$ which does not meet $K_0$
 and such that for any $2$-simplex   $\Delta$ of the triangulation $\Gamma$ the images $\phi _{\epsilon}  (\partial \Delta)$ and $\phi _2 (\Delta)$ are contained in one set $V^z$.   Applying Step 4 and Step 6, we
can extend $\phi _{\epsilon}$ from the boundary $\partial \Delta$ of any $2$-simplex $\Delta$
such that this extension lies  inside the same open set $V^z$ and intersects $K_0$ only in a finite set of points.
Taking all these extensions together, we obtain  the required disk $\phi _{\epsilon}$.

\emph{Step 8}:  The disjoint disk property holds in $W$.

  Thus, let $\phi_1,\phi_2 : D\to W$  and  $\epsilon >0$ be given. Apply  the previous Step 4
 and obtain a map  $\tilde \phi _1 :D \to W$ with distance at most $\frac \epsilon 2$ to $\phi _1$, whose image intersects
 $K_0$ only in a finite set of points $Q=\{x_1,....,x_l \}$.

 We find a covering of the compact image $\phi_ 2 (D)$ by finitely many open neighborhoods $V^z$  as above
  of diameter smaller than $\frac \epsilon 2$,  such that any subset $V^z$ contains at most one of the points $x_i$.

We find a triangulation of the disk $D$ by a finite graph $\Gamma$,
 such that for any  $2$-simplex  $\Delta$ of the triangulation, the image $\phi _2 (\Delta)$ is contained in one of these sets  $V^z$. Arguing as in the previous Step 7 (applying Step 2), we   find  a map $\tilde \phi_2 :\Gamma \to W$ which does not meet $K_0$
 and such that for any $2$-simplex   $\Delta$ of the triangulation the image $\phi _2 (\partial \Delta)$ is contained
 in one of the sets $V^z$.


 By Step 3 and Step 6, for any of our sets $V^z$, the complement   $V^z\setminus Q$ is simply connected.
 Therefore, we can extend $\tilde \phi _2 :\Gamma \to  W$ to a map $\tilde \phi _2:  D\to W \setminus Q$ such
 that $\tilde \phi _2  (\Delta)$ and $\phi _2 (\Delta)$ are in the same set $V^z$ of our covering.

 By construction, the intersection $\tilde \phi _2 ( D) \cap \tilde \phi _1 ( D)$ is contained in the set of manifold points $U\setminus K_0$. Since in the $n$-manifold $U\setminus K$ the disjoint disk property holds true,
  we can
 slightly perturb  $\tilde \phi _2 $ and $ \tilde \phi _1$ (outside of $K$), so that the arising disks do not intersect.

 This finishes the proof of Step 8 and therefore of the Theorem.
\end{proof}

\subsection{Main theorems}
We  now finish the proof of the main theorems.

\begin{proof}[Proof of Theorem \ref{thm: quinn}]
Let  $X$ be a metric space with an upper curvature bound, which is a homology $n$-manifold. Then $X$ is a  GCBA space, by Lemma \ref{lem: homchar}.
We cover $X$ by tiny balls $O$, and apply   Theorem \ref{thm: maintech} in the case $k=n$ and the constant  map $F:O\to \R ^0 =\{0 \}$.
We deduce that $X$ is a topological manifold outside a discrete set of points.
\end{proof}

\begin{proof}[Proof of Theorem \ref{thm: topreg}]
We have seen in Lemma \ref{lem: 1to3} and Lemma  \ref{lem: 2to3}  that (1) implies (3) and that (2) implies (3).

Assume now that (3) holds, thus all spaces of directions are homotopy equivalent to a non-contractible space.
We have seen in Lemma \ref{lem: same}  that $X$ must be a homology $n$-manifold and all spaces of directions are homotopy equivalent to
$\mathbb S^{n-1}$.

By Theorem \ref{thm: maintech}, $X$ is a topological manifold if $\dim (X) \leq 3$.

Let the dimension of $X$ be at least $4$. Then all $\Sigma _x$ are simply connected, hence so are all small
punctured balls $B_{r} ^{\ast} (x)$. 
 The result that $X$ is a topological manifold follows now directly as a combination of  Theorem \ref{thm: quinn} and Theorem \ref{thm: bcl}.




Thus (3) implies (1).

It remains to  prove that (1) implies (2).
 Assuming that $X$ is a topological $n$-manifold let $x\in X$ be arbitrary. 
 We deduce from
  Corollary \ref{cor: homlink}  that any space of directions $\Sigma _x$ is a homology $(n-1)$-manifold and any  tangent space $T_x=C(\Sigma _x)$ is a homology $n$-manifold.  Moreover, any space of direction is homotopy equivalent to $\Sph ^{n-1}$ by Lemma \ref{lem: 1to3}.
  
   If $n\leq 3$ then $\Sigma _x$ is a topological manifold,  Theorem \ref{thm: maintech}, homeomorphic to $\Sph ^{n-1}$ by   Lemma  \ref{lem: 1to3}. Thus,  $T_x$ is homeomorphic to $\R^n$.


  Assume $n\geq 4$. By Theorem \ref{thm: quinn},
 the  set of non-manifold points of $T_x$  is discrete. Due to the conical structure, this directly implies
   that $T_x\setminus \{ 0\}$
   is a topological manifold.   But $\Sigma _x$ and, therefore, all  punctured balls around $0$ in $T_x$  are simply connected.  From Theorem \ref{thm: bcl} we deduce that $T_x$ is a topological  $n$-manifold.

   Thus, $T_x$ is a contractible $n$-manifold, simply connected at infinity, since $\Sigma _x$ is simply connected.
   Therefore, $T_x$ is
   homeomorphic to  $\R^n$.
\end{proof}

\subsection{Some improvements}
Theorem \ref{thm: topreg} can be slightly strengthened in  dimensions $\leq 4$. The first  of these results is
contained in \cite{thurston}.

\begin{thm} \label{thm: dim3}
Let $X$ be a locally compact space with an upper curvature bound.
If $X$ is a homology manifold with $n\leq 3$ then $X$ is a topological manifold.
If $X$ is a topological $n$-manifold with $n \leq 4$ then any space of directions $\Sigma _x$ in $X$
is homeomorphic to $\Sph^{n-1}$.
\end{thm}

\begin{proof}
The first statement is local and is therefore contained (as the case $k=n=3$) in Theorem \ref{thm: maintech}.

To prove the second statement, we use Lemma \ref{cor: homlink} and Lemma \ref{lem: 1to3} to deduce that $\Sigma _x$
is a homology $(n-1)$-manifold, homotopy equivalent to $\Sph^{n-1}$. By the first statement and the resolution of the Poincar\'e conjecture,
$\Sigma _x$ is homeomorphic to $\Sph^{n-1}$.
\end{proof}

Theorem \ref{thm: maintech} and therefore, Theorem \ref{thm: topreg}  can be strengthened as follows.
 Since the result is not used in the sequel, the proof will be somewhat sketchy.
 For definitions and fundamental results about ends of manifolds we refer to \cite{Siebenmann} and  \cite{HR}.
\begin{thm} \label{thm: addendum}
Under the assumptions of Theorem \ref{thm: maintech}, for any $x\in \Pi$ there exists a neighborhood  of $x$ in $\Pi$ homeomorphic
to the open cone $C(M)$ over a topological  $(k-1)$-manifold $M$, with the homology of $\Sph^{k-1}$.
\end{thm}

\begin{proof}
 	We proceed by induction on $k$. 
	
In the cases $k\leq 3$ the statement is clear, since $\Pi$ is a topological manifold. In the case $k=4$ one argues in the same way as in the case $k=3$
in the proof of Theorem \ref{thm: maintech}. Using that the fibers of the Hurewicz fibration $f:B_r^{\ast} (x) \to (0,r)$ are topological $3$-manifolds,
one concludes that $B_r^{\ast} (x)$ is homeomorphic to $(0,r)\times M$ for a topological $3$-manifold $M$.  Thus
 $B_r(x)$ is homeomorphic to the cone $C(M)$.  Since $X$ is a homology $4$-manifold, $M$ must be a homology $3$-sphere.

Let now $k \geq 5$.   Find a small number $r>0$, such that $F$ extends to an $(n-k+1,12\cdot \delta )$-strainer map $\hat F=(F,f)$
on a neighborhood $V$ of $N:=B_r^{\ast} (x) \cap \Pi$ in $X$.  Here $f$ denotes as before, the distance function $f=d_x$.
By Theorem \ref{thm: maintech}, $N$ is a topological $k$-manifold
and, as we have seen in the proof of Theorem \ref{thm: maintech}, the  map $f:N\to (0,r)$ is a Hurewicz fibration.

    We claim that the end of the manifold $N$  corresponding to the point $x$ is \emph{collared}.  Thus,  $N$ contains
a subset homeomorphic to $M\times [0,\infty )$, whose closure in $\Pi$ contains a neighborhood of $x$, for some manifold $M$.
Since $\Pi$ is a homology manifold, this would imply that $M$ must have the homology of $\Sph ^{k-1}$ and finish the proof of the theorem.

For $k\geq 6$ the statement is a direct consequence of Siebenmann's theorem on collared ends, \cite{Siebenmann}, \cite[Theorem 10.2]{HR}, and the observation that
$N$ homotopically retracts onto any compact fiber of $f$.

The following elegant argument due to Steven Ferry covers the case $k=5$ as well as the case $k\geq 6$.

Fix a fiber $\Pi_t =f^{-1} (t)$ for some $t$.  By induction,   $\Pi_t$ is  a   homology $(k-1)$-manifold  with  a finite set $K$ of
 singularities (each of whom has a neighborhood in $\Pi_t$ homeomorphic to a cone).  By the main result of  \cite{Quinn-obstr},
 there exists a
\emph{resolution} $g:M \to \Pi_t$ which is a homeomorphism outside the preimages $g^{-1}  (K)$.

Consider the space $N^+$ obtained by  gluing the cylinder $M\times (-1,1]$ to  $f^{-1} (0,t) \subset N$ by identifying
$M\times \{1 \}$ with $\Pi_t$ along the map $g$.  The space $N^+$ is by construction a topological manifold outside the finitely many singularities in $K\subset \Pi _t$.  Computing the local homology at points in $K$, we see that $N^+$ is a homology $k$-manifold.
Finally, arguing as in the proof of Theorem \ref{thm: maintech}, we see that all points in $K$ have arbitrary small simply connected punctured neighborhoods in $ N^+$.  Applying Theorem \ref{thm: bcl}, we conclude that  $N^+$  is a topological $k$-manifold.

Now we apply, the main theorem from \cite{open-collar}, and see that  the topological $k$-manifold with boundary
$N^+\setminus M\times (-1,0)$  (the boundary is $M\times \{0 \}$)   is homeomorphic to  $M\times [0,\infty )$.
Thus the end of the manifold $N$ is collared.
\end{proof}

\section{Limits of manifolds} \label{sec: limit}
\subsection{Topological stability}
We start with  a part of  Theorem \ref{thm: riemlimreg}:
\begin{thm} \label{thm: part1}
Let a sequence of  complete $\CAT(\kappa)$  Riemannian manifolds $M_i$  of dimension $n$ converge in the pointed Gromov--Hausdorff topology
to a space $X$.  Then $X$  is  a topological  $n$-manifold.
\end{thm}

\begin{proof}
 $X$ is  a GCBA space, \cite[Example 4.3]{LN}.  Due to Theorem \ref{thm: topreg}, it  suffices to prove that for all
 $x\in X$ the space of directions $\Sigma _x$ is  homotopy equivalent to $\Sph ^{n-1}$.  Due to
 Proposition \ref{prop: sph}, it suffices to prove that for all $r$ small enough, the distance sphere
 $S_r (x)$ is homotopy equivalent to $\Sph ^{n-1}$.

  Fix a sequence $x_i  \in X_i$ converging to $x$.  For all $r$ small enough, the injectivity radius of   the Riemannian manifold
  $X_i$ is larger than $r$, hence  the distance sphere $S_r (x_i)$ is homeomorphic to $\Sph ^{n-1}$.  According to Proposition \ref{prop:  sph}, the spheres $S_r (x_i)$ are homotopy equivalent to $S_r (x)$, for all $i$ large enough. This proves the claim.
  \end{proof}

 The $\alpha$-approximation theorem (Theorem  \ref{thm: alpha}) and Petersen's stability theorem (Theorems  \ref{thm: petersen}) give us:
\begin{cor}
Under the assumptions of Theorem \ref{thm: part1}, assume in addition that $X$ is compact.  Then $M_i$ is homeomorphic to $X$, for all $i$ large enough.
\end{cor}

\subsection{Iterated spaces of directions}
In order to prove the remaining statement in Theorem \ref{thm: riemlimreg},
we need to understand spaces of directions of spaces of directions.
For a GCBA space $X$, we call any space of directions $\Sigma _x X$ of $X$ a \emph{first order space of directions} of $X$. Inductively we define
a \emph{$k$-th iterated space of directions of $X$} to be a space of directions $\Sigma _z Y$ of a $(k-1)$-th iterated  space of directions $Y$  of $X$.


Using Theorem \ref{thm:  topreg} we can easily derive the following lemma, clarifying  the second statement in
Theorem \ref{thm: riemlimreg}.
\begin{lem} \label{lem: limits}
Let $X$ be a  locally compact space with an upper curvature bound.  Then the following are equivalent:
\begin{enumerate}
\item $X$ is a topological manifold and, for any $1\leq k <n$, all $k$-th  iterated spaces of directions of $X$ are homeomorphic to  $\mathbb S^{n-k}$.
\item For any $1\leq k <n$, all $k$-th iterated spaces of directions of $X$ are homotopy equivalent to $\Sph^{n-k}$.
%
\end{enumerate}
\end{lem}

\begin{proof}
Clearly (1) implies (2).

Assuming (2), we deduce from Theorem \ref{thm: topreg} that $X$ is a topological manifold.
Thus, $X$ is a GCBA space and all of its iterated spaces of directions are compact $\CAT(1)$ spaces.
Let $\Sigma $ be a $k$-th iterated space of directions of $X$.  By assumption, all of its spaces of directions
are homotopy equivalent to $\Sph^{n-k-1}$. By Theorem \ref{thm: topreg}, the space $\Sigma $ is
a topological manifold. Due to the resolution of the (generalized) Poincar\'e conjecture,
$\Sigma $ is homeomorphic to $\Sph ^{n-k}$.
\end{proof}

Iterated spaces of directions can be seen in factors of blow-ups of the original space.
More generally, we have:

 \begin{lem} \label{cor: limoflim}
Let  $X_i$ be complete GCBA spaces which are $\CAT(\kappa)$ for a fixed  $\kappa$.  Assume that
$(X_i,x_i)$ converge in the pointed Gromov--Hausdorff topology to a  GCBA space $(X,x)$.

Then, for any non-empty $k$-th iterated space of directions $\Sigma ^k$ of $X$, there exists
a sequence $z_i\in X_i$
and a sequence  $t_i \geq 1$, such that, possibly after choosing a subsequence,
we have the following convergence in the pointed Gromov--Hausdorff topology:
$$(\R^{k-1} \times C\Sigma ^k,0) =  \lim _{i\to \infty}   (t_i\cdot X_i, z_i) \, .$$
\end{lem}


\begin{proof}
   Consider  the set  $\mathcal  L$ of (isometry classes of) all pointed locally compact spaces $(Y,y)$ which can be obtained as a  pointed Gromov--Hausdorff limit of a  subsequence of a sequence $(t_i\cdot X_i,y_i)$, for some    $y_i\in  X_i$ and some sequence $t_i \geq 1$.

    The set $\mathcal L$ consists of complete GCBA spaces,  it contains the space $(X,x)$. With any space
    $(Y,y)$, the family $\mathcal L$ contains the space $(Y,y')$, for any $y' \in Y$. Thus, we may ignore the base points.
      Moreover,
    $\mathcal L$  is closed under rescaling  with numbers $t\geq 1$ and under pointed Gromov--Hausdorff convergence.  Thus, with every space $Z$ the family  $\mathcal L$ contains any of the  tangent spaces $T_zZ$.

 For any   $k\geq 1$ and  any non-empty $k$-th  iterated space of directions   $\Sigma ^k$ of $X$, we need to prove that 
    $Z= \R^{k-1} \times C\Sigma^k$ is contained in $\mathcal L$.

    We proceed by induction on $k$. The case $k=1$, thus $C\Sigma ^1$ being a tangent cone of $X$ at some point is
    already verified.

    Assume that we have already  verified the claim for $k$. Let $\Sigma ^k$ be
     any $k$-th iterated space of directions  of $X$  and let $v\in \Sigma ^k$   be arbitrary,
     such that $\Sigma := \Sigma _v \Sigma ^k$ is not empty.

     By the inductive assumption, the space $Z= \R^{k-1} \times C\Sigma^k$ is contained in $\mathcal L$.
    Then, also  $T_vZ=\R ^k \times C  \Sigma$ is contained in $\mathcal L$.
    This verifies the claim for $k+1$ and finishes the proof of the lemma.
 \end{proof}

\subsection{Limits of Riemannian manifolds}
Now we are in position to formulate and to prove the  following generalization of the remaining part of 
Theorem \ref{thm: riemlimreg}. Its proof relies on some stability properties of strainer maps.

\begin{thm}
 Under the assumptions of Theorem \ref{thm: riemlimreg}, for any $k\leq n$,  any $k$-th iterated   space of directions of $X$  is
homeomorphic $\Sph ^{n-k}$.

%
%
\end{thm}

\begin{proof}
 It  suffices  to prove that, for all
$ k < n$, any $k$-th iterated space of directions $\Sigma ^k$ of $X$ is
homotopy equivalent to $\Sph^{n-k}$,  Lemma \ref{lem: limits}.

 Let us fix such $\Sigma =\Sigma ^k$.  Due to Corollary \ref{cor: limoflim} we get (by rescaling the manifolds $M_i$)
 a sequence of pointed Riemannian manifolds  $(N_i,p_i)$ with the following properties.
 The manifold $N_i$ is $\CAT (\kappa_i)$, with $\kappa _i $ converging to $0$, and the sequence $(N_i,p_i)$ converges in the pointed Gromov--Hausdorff topology to $Y= (\R^{k-1} \times C\Sigma, 0)$.

 We fix the standard coordinate vectors $e_1,....,e_{k-1} \in \R^{k-1}  \times \{0\} \subset Y$  and consider the
 map $F:Y\to \R^{k-1}$ whose coordinates $f_1,...,f_{k-1}$ are the distance functions to the points $e_j$.
 The geodesics from $e_j$ to $0$ do not branch at $0$. By
 definition of strainer maps, \cite[Sections 7, 8]{LN},
 for every $\delta >0$ there exists some $\epsilon >0$ such that $F$ is a $(k-1,\delta)$-strainer map
 in the ball $B$ of radius $\epsilon$ around $0$ in $Y$.  We fix $\delta < \frac 1 {20\cdot  k^2}$ and $\epsilon $ as above.

 We take  $(k-1)$-tuples of points in $N_i$ converging to  $(e_1,...,e_{k-1})$ and consider the correspondingly defined maps $F_i:N_i \to \R^{k-1}$ which converge to $F$.  By the openness property of strainers the following holds true, \cite[Lemma 7.8]{LN}:
 For all  $i$ large enough,  the map $F_i$ is a $(k-1,\delta)$-strainer in the ball  $B^i$ of radius $\epsilon$ around $p_i$.

 Denote by $\Pi $ the fiber of $F$ in $B$  through $0$ and by $\Pi_i$ the fiber  of $F_i$ in $B^i$ through  $p_i$.
The fibers $\Pi_i$ are locally uniformly contractible. More precisely,	
by \cite[Lemma 7.11, Theorem 9.1]{LN}, there exists some $\epsilon _1 >0$ with the  following property.  For all $i$ large enough,
any $q\in \Pi_i $ and any $r<\epsilon _1$ such that  $\bar B_r (q) \cap \Pi _i$ is compact, this compact set is contractible.

  Denote by $g:\Pi \to \R$ and by $g_i:\Pi_i \to \R$ the distance functions to the points $0$ and $p_i$ respectively.

 By the extension property of strainers, we may assume, after making $\epsilon $ smaller, that  the map
 $\hat F =(F,g) :V\to \R^k $ is a $(k,12\cdot \delta )$-strainer map on an open  neighborhood $V$ of
 $B_{\epsilon} ^{\ast} (0) \cap \Pi$ in $Y$.    The fibers of the map $\hat F$ through points on $\Pi$ are (compact)
 distance spheres in $\Pi$ around $0$.

 From the homotopy  stability of fibers of strainer maps  \cite[Theorem 13.1] {LN},  for any $0<r<\epsilon$ there exists
 some $i_0$ such that for all $i>i_0$ the following holds true.  The distance sphere $S_r (p_i) \cap \Pi_i$ is compact
 and homotopy equivalent to the distance sphere $S_r (0) \cap \Pi$.

 We subdivide the rest of the proof in 6 steps.

 \emph{Step 1}:  The  manifold $N_i$ has injectivity  radius larger than $2$, for all large $i$.   The strainer map $F_i:B^i \to \R^{k-1}$ is smooth. The distance function $g_i : B^i \to \R$ is smooth outside $p_i$.

Indeed, the first statement is a consequence of our assumption that $N_i$ is $\CAT(\kappa _i)$ with $\kappa _i$ converging
to $0$.  The remaining statements follow from the first one.

\emph{Step 2}: The strainer map $F_i :B^i \to \R^{k-1}$ is a submersion. Thus $\Pi_i$ is a smooth submanifold of $N_i$.

Indeed, the strainer map $F_i$ is a    $2\sqrt k$-open map, see \cite[Lemma 8.2]{LN}.
In particular, the  differential  of $F_i$ at any point of $B^i$  is surjective.
This implies the first and, therefore, the second claim.

\emph{Step 3}: There exists $0<\epsilon _0 <\frac {\epsilon_ 1} {2}$ such that,  for all $i$ large enough, 
the map $g_i :\Pi _i \cap B_{\epsilon _0} ^{\ast} (p_i)  \to \R$ has at all points a gradient of norm between $\frac 1 2$ and $1$,
with respect to the intrinsic metric of $\Pi_i$.

Indeed, for any $x\in B _{\epsilon }^{\ast}  (p_i)$ the gradient $\nabla _x g_i$ of the map $g_i : B^i \to \R$ is the unit velocity vector of the geodesic connecting $p_i$ with $x$.   Thus, for every $x\in \Pi_i \setminus \{ p_i \}$ the gradient of
$g_i$ at $x$ with respect to the induced metric of $\Pi _i$  is the projection of $\nabla _x g_i$ to the tangent space
$T_x \Pi_i$.

There exists some $\epsilon _0>0$ such that for all $i$ large enough and all $x\in \Pi _i \cap B_{\epsilon _0} ^{\ast} (p_i) $
the inequality $$|DF_i (\nabla _x g_i ) | \leq k\cdot 2\cdot \delta   < \frac 1 {10\cdot  k} \, ,$$
due to \cite[Lemmas 7.6, 7.10, 7.11]{LN}.

Since the  differential  of $F_i$ at $x$  is $2\sqrt k$-open this implies that the projection of
$\nabla _x g_i$ to the tangent space
$T_x \Pi_i$ has norm at least $\frac 1 2$.

\emph{Step 4}: In the notations above, for all $i$ large enough and all $0<r<\epsilon _0$,
 the distance sphere $S_r (p_i)= g_i^{-1} (r) \subset  \Pi_i$  is diffeomorphic to $\Sph^{n-k}$.
 Moreover, $S_r (p_i)$ is locally uniformly contractible with respect to the
 contractibility function $\rho :[0,r) \to \R$ given by $\rho (s)= 2s$.

 Indeed,  for all sufficiently small $r$ (depending on $i$), the fact that $S_r (p_i)$ is diffeomorphic to a sphere
 is true for  every smooth submanifold  of a smooth Riemannian manifold, as easily seen in local coordinates.

 The fact that for all $r<\epsilon_0$ the level sets  of $g_i$ are diffeomorphic among each other is a consequence of the (easy part) of
 Morse theory, since $g_i$ has no critical points in $\Pi_i$.

 Finally, the gradient flow of the function $g_i $ on $\Pi_i$ retracts $\Pi  _i \setminus \{p_i \}$ onto $S_r (p_i)$. 
   Moreover, along this retraction,  any point moves with velocity less than $1$ and the distance to $S_r (p_i)$ decreases with velocity at least $\frac 1 2$. Thus, for any point $q\in S_r (p_i)$ and any $s<r$ the retraction sends the  ball $\bar B_s (q) \cap \Pi_i$ into the ball $\bar B_{2s} (q) \cap S_r (p_i)$.
 
 Since the ball $\bar B_s (q) \cap \Pi_i$  is contractible, we deduce that the ball $\bar B_s (q) \cap S_r (p_i)$ is contractible inside the ball $\bar B_{2s} (q) \cap S_r (p_i)$.  Finishing the proof of Step 4.

\emph{Step 5}:  For every $r<\epsilon _0$ the distance sphere $S_r (0) \cap \Pi$ is homotopy equivalent 
to $\Sph ^{n-k}$. Moreover, $S_r (0) \cap \Pi$  is uniformly locally contractible with respect to the contractibility function
$\rho :[0,r) \to \R$ given by $\rho (s)= 2s$.

Indeed, the  sets $S_r (p_i)$ converge to $S_r(p)$ in the Gromov--Hausdorff topology. Hence the result follows from
\cite[Section 5]{Petersen}.

\emph{Step 6}:  The space $Y$ is a cone, hence invariant under rescalings. However, the rescaled sequence
  $(m\cdot \Pi , 0) \subset Y$ converges to the factor $(C\Sigma,0)$, for $m\to \infty$.

Under this convergence the rescaled spheres $m \cdot (S_{\frac 1 m} (0)\cap \Pi) $ converge to $\Sigma$.
All these spheres are homotopy equivalent to $\Sph^{n-k}$ and are uniformly locally contractible.
Applying Theorem \ref{thm: petersen} once more, we deduce that $\Sigma $ is homotopy equivalent to $\Sph ^{n-k}$.
\end{proof}

\section{A sphere theorem} \label{sec: last}
\subsection{Pure-dimensional spaces} In order to deduce Theorem \ref{thm: newchar} from
 Theorem \ref{thm: sphere}, we need to show  that a GCBA space all of whose
 spaces of directions are spheres (of a priori different dimensions)  must be a manifold. We address
 this question in a slightly more general setting.
 
 We define a GCBA space $X$ to be \emph{purely $n$-dimensional} if all of its non-empty open subsets
 have dimension $n$.  We say that $X$ is \emph{pure-dimensional} if $X$ is purely $n$-dimensional for some $n$.

 Due to \cite[Corollary 11.6]{LN}, a  GCBA space $X$ is purely $n$-dimensional if and only if all of its tangent spaces 
 $T_xX$ have dimension $n$. This happens if and only if all spaces of directions $\Sigma _xX$ have dimension $n-1$. 
 Using the stability of dimension under convergence proved in \cite{LN}, we can now show:
\begin{prop}  \label{prop: pure}
A connected GCBA space $X$  is pure-dimensional if and only if all tangent  spaces of 
 $X$
are pure-dimensional.
\end{prop}

\begin{proof}
Let  $X$ be purely $n$-dimensional  and $x\in X$ arbitrary. Applying \cite[Lemma 11.5]{LN} to the  convergence of the rescaled balls in $X$ around $x$ to $T_xX$ we deduce, that for any $v\in T_xX$ the dimension of $T_v (T_xX)$ is $n$.
Thus, $T_xX$ is purely $n$-dimensional.  

 Assume that all spaces of directions of $X$ are pure-dimensional. By the connectedness of $X$ it suffices to prove 
 that every point $x\in X$ has a pure-dimensional open neighborhood. Therefore, we may replace $X$ by a small
 ball around some of its points and assume that $X$ is a geodesic space  and  has finite dimension $n$.
  Consider the set $X^n$ 
 of all points in $X$ for which $T_xX$ has dimension $n$.  By  \cite[Corollary 11.6]{LN}  the set $X^n$ is closed in $X$.
   
 Assume that $X^n$ is not $X$.  Then we find a point $y\in X\setminus X^n$ such that there exists a point $x\in  X^n$
 closest to $y$ among all points of $X^n$.  Consider the geodesic $\gamma$ from $x$ to $y$ and set 
 $x_i=\gamma (\frac 1 i)$.   Then the ball $B_{\frac 1 i} (x_i)$ does not intersect $X^n$, hence has dimension less than $n$.  Under the convergence of $( i\cdot X, x)$ to $T_xX$, the closed balls $i\cdot \bar B_ {\frac 1 i} (x_i)$  in $X$ converge to the closed ball of radius $1$  around the starting direction $v=\gamma '(0)\in \Sigma_x \subset T_x$.
 
 Applying \cite[Lemma 11.5]{LN} again, we see that the open ball $B_1(v)$ in $T_x$ has dimension less than $n$.
 Since $T_xX$ is $n$-dimensional, this contradicts the assumption that $T_xX$ is pure-dimensional.  This contradiction
 shows $X=X^n$ and finishes the proof.
\end{proof}

\subsection{The conclusions}
The proof of Theorem \ref{thm: sphere} relies also on the following observation, well-known to experts. We could not find a reference and include a short proof.
\begin{lem} \label{lem: cat}
If a closed topological  $n$-manifold $M$  is  covered by two contractible open subsets $U,V$
then $M$ is homeomorphic to $\Sph ^n$.
\end{lem}

\begin{proof}
The assumption  implies that for any commutative ring 
$R$, the cup product of any two elements in the reduced cohomology $ H^{\ast} (M,R)$ is $0$,
\cite[Section 3.2, Exercise 2]{Hatcher}.   

By the Poincar\'e duality we deduce that $M$ has the same cohomology with $R$-coefficients as $\Sph ^n$ 
if $M$ is $R$-orientable. Applying this for $R=\mathbb Z_2$ we deduce that $H_{n-1} (M,\mathbb Z_2)=0$.
Therefore, $M$ is orientable, \cite[Chapter 3, Corollary 3.28]{Hatcher} with respect to integer coefficients and 
has the same integer homology and cohomology as $\Sph ^n$. 

 By the theorem of Mayer--Vietoris, $0=H_1(M)=H_0(U\cap V)$. Thus, the intersection $U\cap V$ is connected. 
 Applying van Kampen's theorem, we deduce that $M$ is simply connected. By the theorem of  Whitehead, $M$ 
 is  homotopy equivalent to $\Sph ^n$.  By the resolution of the (generalized) Poincar\'e conjecture, $M$ is homeomorphic to $\Sph ^n$.
\end{proof}

Now we can finish:

\begin{proof}[Proof of Theorem \ref{thm: sphere}]
We proceed by induction on the (always finite) dimension of the compact GCBA space $\Sigma$. 

By assumption, $\Sigma$ is $\CAT(1)$ and the cone $C\Sigma$ is a geodesically complete $\CAT(0)$ space.

 If the dimension of $\Sigma $ is $0$, then $\Sigma $ is discrete and not a singleton.
All  points in $\Sigma $ have distance at least $\pi$ from each other.  The assumption on the triples of points
implies that $\Sigma $ has exactly two points. Hence $\Sigma $ is homeomorphic to $\Sph ^0$.

Assume now that the statement is proven for all spaces  of dimension less than $n$ and let 
$\Sigma $ be $n$-dimensional.  Let $x\in \Sigma $ be a point. Then the space of directions $\Sigma _x$ is
a GCBA space of dimension less than $n$.  If there exists a triple of points $v_1,v_2,v_3$ in $\Sigma _x$ with pairwise distances at least $\pi$ then we consider geodesics $\gamma _i$ in $\Sigma $ starting in $x$ in  the directions of 
$v _i$ (which exist by the geodesical completeness, see \cite[Section 5.5]{LN}).  Then the points $x_i=\gamma _i (\frac \pi 2)$ have in $\Sigma $ pairwise distances $\pi$, in contradiction to our assumption.

Thus, by the inductive assumption, each space of directions $\Sigma _x$ is homeomorphic to some sphere.
Therefore, all spaces of directions $\Sigma _x$ in $\Sigma $ and hence all tangent spaces $T_x$ are pure-dimensional.
Due to Proposition \ref{prop: pure}, the space $\Sigma $ must be purely $n$-dimensional. Then, all spaces of directions
$\Sigma _x$ are $(n-1)$-dimensional, hence homeomorphic to $\Sph ^{n-1}$ by the inductive hypothesis.

By Theorem \ref{thm: topreg}, the space $\Sigma $ is a topological $n$-manifold.  

Consider a pair of points $x,y \in \Sigma $ at distance $\pi$. By the assumption on triple of points,
 there are no points $z \in \Sigma$ with distance
at least $\pi$ to $x$ and $y$. Therefore, all of $\Sigma$ is contained in the union of the two balls
$B_{\pi} (x)$ and $B_{\pi} (y)$.    By the $\CAT(1)$ assumption, both balls are contractible. 
 Thus,  $\Sigma$ is  homeomorphic to $\Sph ^n$, by 
Lemma \ref{lem: cat}.
\end{proof}

\begin{proof}[Proof of Theorem \ref{thm: newchar}]
Let  $X$ be  a connected GCBA space that does not contain an isometrically embedded tree different from 
an interval.  

Let $x\in X$ be a point. If there is a triple of points $v_1,v_2,v_3$ in $\Sigma _x$ with pairwise distances at least $\pi$
then  we obtain an isometrically embedded tree by taking the union of 3 short geodesics $\gamma _i$ starting in the direction of $v_i$, in contradiction to
our assumption.  By  Theorem \ref{thm: sphere}, $\Sigma_x$ must be homeomorphic to some sphere.

  As in the first part of the proof of
Theorem \ref{thm: sphere}, we  now deduce from Proposition \ref{prop: pure}  and Theorem \ref{thm: topreg} 
that $X$ is a topological manifold.
\end{proof}

 Finally, from Theorem \ref{thm: sphere} and the optimal lower bound on the volume of balls, \cite[Proposition 6.1]{N2}, we  deduce:

	\begin{thm} \label{thm: volumesph}
	Let $\Sigma$ be a purely $n$-dimensional,
	compact, locally geodesically complete $\CAT(1)$ space.
	If $\mathcal{H}^n (\Sigma) < \frac{3}{2} \cdot \mathcal{H}^n (\Sph^n)$,
	then $\Sigma$ is homeomorphic to $\Sph^n$,
	where $\mathcal {H}^n$ is the $n$-dimensional Hausdorff measure.
\end{thm}

\begin{proof}
	Otherwise, $X$  contains a triple of points at pairwise distances at least $\pi$.  The open balls of radius $\frac \pi 2$
	around these points are disjoint.  Each of these balls has the $\mathcal H^n$-measure  not less than $\frac 1 2 \cdot \mathcal {H}^n (\Sph^n)$.  This  contradicts the prescribed upper volume bound of $X$.
\end{proof}

\bibliographystyle{alpha}
\bibliography{Reg}

\end{document}